\documentclass[letterpaper,10pt,twocolumn]{ieeeconf}
\usepackage[utf8]{inputenc}

\usepackage{amsmath,mathrsfs,amsfonts,amssymb,graphicx,epsfig}
 \usepackage{amsthm}
\usepackage{subcaption}
\usepackage{color,multirow,rotating}
\usepackage{algorithm,algpseudocode,algorithmicx}
\usepackage{cite}
\usepackage{url}
\usepackage{framed}
\usepackage{bm}

\newcommand{\bs}{\ensuremath{\boldsymbol}}
\newcommand{\tr}[1]{\ensuremath{\mathrm{trace}\big( #1 \big)}}
\renewcommand{\t}{\ensuremath{^{\mathrm{T}}}}

\newcommand{\uff}{\bs{u}_{\rm{ff}}}
\renewcommand{\vec}{{\rm{vec}}}
\newcommand{\jac}{{\rm{D}}}
\newcommand{\differential}{{\rm{d}}}
\newcommand{\hess}{{\rm{Hess}}}

\newtheorem{theorem}{Theorem}

\newtheorem{lemma}{Lemma}
\newtheorem{proposition}{Proposition}
\newtheorem{remark}{Remark}

\newtheorem{example}{Example}

\IEEEoverridecommandlockouts

\title{On the Convexity of Discrete Time Covariance Steering in Stochastic Linear Systems with Wasserstein Terminal Cost}

\author{Isin M. Balci, Abhishek Halder, Efstathios Bakolas
\thanks{Isin M. Balci and Efstathios Bakolas are with the Department of Aerospace Engineering and Engineering Mechanics, University of Texas at Austin, TX 78712-1221, USA; Abhishek Halder is with the Department of Applied Mathematics, University of California, Santa Cruz, CA 95064, USA; Emails: {\tt\small{isinmertbalci@utexas.edu, ahalder@ucsc.edu, bakolas@austin.utexas.edu}}%
}}

\begin{document}

\maketitle

\begin{abstract}
    In this work, we analyze the properties of the solution to the covariance steering problem for discrete time Gaussian linear systems with a squared Wasserstein distance terminal cost. In our previous work~\cite{p:balci2020covariance}, we have shown that by utilizing the state feedback control policy parametrization, this stochastic optimal control problem can be associated with a difference of convex functions program. Here, we revisit the same covariance control problem but this time we focus on the analysis of the problem. 
    Specifically, we establish the existence of solutions to the optimization problem and derive the first and second order conditions for optimality. We provide analytic expressions for the gradient and the Hessian of the performance index by utilizing specialized tools from matrix calculus. Subsequently, we prove that the optimization problem always admits a global minimizer, and finally, we provide a sufficient condition for the performance index to be a strictly convex function (under the latter condition, the problem admits a unique global minimizer). In particular, we show that when the terminal state covariance is upper bounded, with respect to the L\"{o}wner partial order, by the covariance matrix of the desired terminal normal distribution, then our problem admits a unique global minimizing state feedback gain. 
    The results of this paper
    set the stage for the development of specialized control design tools that exploit the structure of the solution to the covariance steering problem with a squared Wasserstein distance terminal cost.
\end{abstract}

\section{Introduction}
In this work, we study the existence and uniqueness of solutions to the covariance steering problem for discrete time Gaussian linear systems with a squared Wasserstein distance terminal cost. This instance of stochastic optimal control problem seeks for a feedback control policy that will steer the probability distribution of the state of the uncertain system, close to a goal multivariate normal distribution over a finite time horizon, where the closeness of the two distributions is measured in terms of the squared Wasserstein distance between them. In our previous work~\cite{p:balci2020covariance}, we have shown that the latter problem can be reduced into a difference of convex functions program (DCP) provided that the control policy conforms to the so-called state feedback control parametrization according to which the control input can be expressed as an affine function of the current state and all past states visited by the system. Whereas the focus in~\cite{p:balci2020covariance} was on the control design problem, in this work we focus on the analysis of the problem and in particular, addressing questions about the existence and uniqueness of solutions and the convexity (or lack thereof) of the performance index. 

\textit{Literature review:} Early works on covariance control problems can be attributed to Skelton and his co-authors who mainly examined infinite-horizon problems in a series of papers (refer to, for instance,~\cite{p:skeltonIJC,p:skeltonTAC,p:levy1997discrete}). Recently, finite-horizon covariance control problems for Gaussian linear systems have received significant attention; the reader may refer to~\cite{p:georgiou15A,p:georgiou15B,p:georgiou15C} for the continuous-time case and~\cite{p:bakcdc16,p:PT2017,p:BAKOLAS2018,p:golds2017,p:bakCDC2018,p:ridderhof2020,p:kotsalis2020} for the discrete-time case. The covariance steering problem for continuous-time Gaussian linear systems with a Wasserstein distance terminal cost was first studied in \cite{halder2016finite} whereas the same problem but for the discrete-time case was studied in \cite{p:balci2020covariance}. Both of these references present numerical algorithms (shooting method in \cite{halder2016finite} and convex-concave procedure in \cite{p:balci2020covariance}) for control design but do not address theoretical questions regarding the existence and uniqueness of solution, or investigate convexity properties of the performance index. 

\textit{Main contributions:}
Next we summarize the main contributions of this paper. First, we establish the existence of at least one global minimizer to the optimization problem. Subsequently, we derive first and second order 
conditions of optimality, and provide analytic expressions for the gradient and the Hessian of the performance index by utilizing specialized tools from matrix calculus (these analytic expressions may also facilitate the implementation of numerical optimization algorithms, and thus improve in practice the speed of convergence). Finally, we present a sufficient condition for the performance index to be a strictly convex function under which the optimization problem admits a unique solution. In particular, we show that when the terminal state covariance is bounded from above, with respect to the L\"{o}wner partial order over the cone of positive semidefinite matrices, by the covariance matrix of the goal normal distribution, then the Hessian of the performance index becomes a strictly positive definite matrix, which in turn implies that the performance index is a strictly convex function. 


\textit{Outline of the paper:} In Section \ref{s:prelim}, we review a few important results from matrix calculus that we use throughout the paper. In Section \ref{s:problem}, we formulate the covariance steering problem with a Wasserstein distance terminal cost, and briefly outline its reduction into a DCP. Sections \ref{sec:FirstOrderConditions} and \ref{sec:SecondOrderConditions} present the first and second order optimality conditions for the latter optimization problem along with a sufficient condition for the convexity of the performance index. Finally, Section~\ref{s:conclusion} concludes the paper with a summary of remarks and directions for future research.

\section{Preliminaries}\label{s:prelim}
Here we collect some notations and background material that will come in handy throughout this paper.
\subsubsection*{Set and inequality notations} We denote the set of nonnegative integers as $\mathbb{N}_{0}:=\{0,1,2,\hdots\}$, and for any positive integer $\nu$, let $\mathbb{N}_{0}[\nu]:=\{0,1,\hdots,\nu\}$. We use the inequalities $\succeq$ and $\succ$ in the sense of L\"{o}wner partial order.
\subsubsection*{Kronecker product, Kronecker sum, and the $\vec$ operator} The basic properties of Kronecker product will be useful in the sequel, including 
\begin{align}
\left(M_{1}\otimes M_{2}\right)\left(M_{3}\otimes M_{4}\right) = \left(M_{1}M_{3}\otimes M_{2}M_{4}\right),    
\label{ProductOfKronProduct}    
\end{align}
and that matrix transpose and inverse are distributive w.r.t. the Kronecker product. The vectorization operator $\vec(\cdot)$ and the Kronecker product are related through
\begin{align}
\vec\left(M_{1}M_{2}M_{3}\right) = \left(M_{3}\t\otimes M_{1}\right)\vec\left(M_{2}\right).
\label{kronANDvec}    
\end{align}
Furthermore, 
\begin{align}
{\rm{trace}}\left(M_{1}^{\top}M_{2}\right) = \vec{(M_{1})}\t \vec{(M_{2})}.
\label{FrobInnerProductANDVec}
\end{align}
We will also need the Kronecker sum 
\[M_{1}\oplus M_{2}:= M_{1}\otimes I + I\otimes M_{2},\]
where $I$ is an identity matrix of commensurate dimension. For matrices $M,L$ of appropriate size and $L$ non-singular, we have 
\begin{align}
(L\otimes L)(M\oplus M)(L^{-1}\otimes L^{-1}) \!=\! LML^{-1}\oplus LML^{-1}
\label{KronSumProductSimilarityTransform}
\end{align}
which is easy to verify using the definition of Kronecker sum and \eqref{ProductOfKronProduct}, and will be useful later.

\subsubsection*{Commutation matrix} The commutation matrix $K_{0}$ is the unique symmetric permutation matrix such that
\[\vec\left(M\right) = K_{0}\:\vec\left(M\t\right),\]
see e.g., \cite{neudecker1983some}. Being orthogonal, $K_{0}$ satisfies
\[K_{0}^{-1} = K_{0}\t = K_{0}.\]
Therefore, $K_{0}$ is idempotent of order two. Two useful properties of $K_{0}$ are
\[K_{0}\:\vec\left(I\right) = \vec\left(I\right), \; K_{0}\left(M_{1}\otimes M_{2}\right) = \left(M_{2}\otimes M_{1}\right)K_{0}.\]
Notice that $K_{0}$ being symmetric orthogonal, its eigenvalues are $\pm 1$. Consequently, the matrix $I+ K_{0}$, which is also symmetric idempotent, has eigenvalues $0$ and $2$.

Another observation that will be useful is that $I+K_{0}$ commutes with ``self Kronecker product or sum", i.e., for any square matrix $M$, we have
\begin{subequations}
\begin{align}
\left(I+K_{0}\right)\left(M\otimes M\right) &= \left(M\otimes M\right)\left(I+K_{0}\right),    
\label{IplusK0commutesWithSelfKronProduct}\\
\left(I+K_{0}\right)\left(M\oplus M\right) &= \left(M\oplus M\right)\left(I+K_{0}\right),    
\label{IplusK0commutesWithSelfKronSum}
\end{align}
\label{IplusK0commutesWithSelfKron}
\end{subequations}
which follows from the property of $K_{0}$ mentioned before. We also have
\begin{align}
\left(I+K_{0}\right)\left(M\oplus M\right)^{-1} = \left(M\oplus M\right)^{-1}\left(I+K_{0}\right).    
\label{IplusK0commutesWithSelfKronSumInverse}    
\end{align}
To see (\ref{IplusK0commutesWithSelfKronSumInverse}), notice that $K_{0}\left(M\oplus M\right)^{\!-1}$ equals
\begin{align*}
&\left(\left(M\oplus M\right)K_{0}^{-1}\right)^{\!-1}\!\!=\! \left(\left(M\oplus M\right)K_{0}\right)^{-1}\!\!=\! \left(K_{0}\left(M\oplus M\right)\right)^{\!-1}\\
&=\left(M\oplus M\right)^{\!-1}K_{0}^{-1} \!=\! \left(M\oplus M\right)^{\!-1}K_{0}.
\end{align*}

\subsubsection*{Matrix differential and Jacobian} The matrix differential $\differential(\cdot)$ and the vectorization $\vec(\cdot)$ are linear operators that commute with each other. We will frequently use the Jacobian identification rule \cite[Ch. 9, Sec. 5]{magnus2019matrix}, which for a given matrix function $F(X)$, is
\begin{align}
\differential\:\vec\left(F(X)\right) = \jac F(X)\:\differential\:\vec\:X,    
\label{JacIdentificationRule}
\end{align}
where $\jac F(X)$ is the Jacobian of $F$ evaluated at $X$. In case $F$ is independent of $X$, the Jacobain $\jac F$ is a zero matrix. Some Jacobians of our interest are collected in the Appendix.

\subsubsection*{Matrix geometric mean} Given two symmetric positive definite matrices $A$ and $B$, their geometric mean (see e.g., \cite{lawson2001geometric}) is the symmetric positive definite matrix
\begin{align}
A\#B := A^{\frac{1}{2}} \left(A^{-\frac{1}{2}} B A^{-\frac{1}{2}}\right)^{\frac{1}{2}} A^{\frac{1}{2}}.
\label{matrixGM}    
\end{align}
It satisfies intuitive properties such as $A\#A=A$, $A\#B=B\#A$, $\left(A\#B\right)^{-1}=A^{-1}\#B^{-1}$.

\subsubsection*{Function composition and normal distribution} We use the symbol $\circ$ to denote function composition. The symbol $z\sim\mathcal{N}(\mu,\Sigma)$ denotes that the random vector $z$ has normal distribution with mean vector $\mu$ and covariance matrix $\Sigma$.


\section{Problem Set up}\label{s:problem}
We consider a discrete-time stochastic linear system
\begin{equation}\label{eq:system-dynamics}
    x_{k+1} = A_k x_k + B_k u_k + G_k w_k, \quad k\in\mathbb{N}_{0},
\end{equation}
where $x_k \in \mathbb{R}^{n_x}$, $u_k\in \mathbb{R}^{n_u}$, and $w_k \in \mathbb{R}^{n_w}$ denote the state, control input, and disturbance vectors at time $t=k$, respectively. It is assumed that the initial state is a normal vector and in particular, $x_0 \sim \mathcal{N}(\mu_0, S_0)$, where $\mu_0\in \mathbb{R}^n$ and $S_0 \succ \bm{0}$, and in addition, the disturbance process is a sequence of independent and identically distributed random vectors $w_k \sim \mathcal{N}(0, S_w)$ for all $k \in \mathbb{N}_{0}$ and $S_w \succ \bm{0}$. We suppose that $x_0$ and $w_k$ are mutually independent for all $k \in \mathbb{N}_0$, from which it follows that $\mathbb{E}[x_0 w_k\t]=0$ for all $k \in \mathbb{N}_0$, where $\mathbb{E}\left[\cdot\right]$ denotes the expectation operator. We assume that the matrices $G_{k}$ are full rank for all $k\in\mathbb{N}_{0}[N-1]$.

For $N\in\mathbb{N}_{0}$, let 
\begin{align*}
\bm{x} &:= [x_0\t, x_1\t, \dots, x_N\t]\t \in \mathbb{R}^{(N+1)n_x},\\
\bm{u} &:= [u_0\t, u_1\t, \dots, u_{N-1}\t]\t \in\mathbb{R}^{N n_u},\\
\bm{w} & := [w_0\t, w_1\t, \dots, w_{N-1}\t]\t \in\mathbb{R}^{N n_w}.
\end{align*}
Then, we can write
\begin{equation}\label{s:simplified-dynamics}
    \bs{x} = \bs{\Gamma} x_0 + \bs{H_u u} + \bs{H_w w},
\end{equation}
where the block (column) vector
\begin{equation}\label{eq:Gammadef}
    \bs{\Gamma}  := [I_{n_x}\t~\Phi\t(1,0)~\Phi\t(2,0)~\dots~ \Phi\t(N,0)]\t,
\end{equation}
and for all $k,n\in\mathbb{N}_{0}$ with $k\geq n$, the matrices $\Phi(k,n):=A_{k-1} \hdots A_{n}$, and $\Phi(n,n):=I$ (note that $\Phi(n+1,n)=A_{n}$). 

Furthermore,
\begin{align}\label{eq:Hudef}
    \bs{H_u} & := \begin{bmatrix}
    \bs{0} & \bs{0} & \dots & \bs{0} \\
    B_0 & \bs{0} & \dots & \bs{0} \\
    \Phi(2,1)B_0 & B_1 & \dots & \bs{0} \\
    \vdots & \vdots & \vdots & \vdots \\
    \Phi(N, 1)B_0 & \Phi(N, 2)B_1 & \dots & B_{N-1}  
    \end{bmatrix}, 
\end{align}
and $\bs{H_w}$ is defined likewise by replacing the matrices $\{B_{k}\}_{k=0}^{N-1}$ in \eqref{eq:Hudef} with the matrices $\{G_{k}\}_{k=0}^{N-1}$.

The problem of interest is to perform minimum energy feedback synthesis for \eqref{eq:system-dynamics} over a time horizon of length $N$, such that the distribution of the terminal state $x_{N}$ goes close to desired distribution $\mathcal{N}\left(\mu_d,S_d\right)$ where $\mu_{d}\in\mathbb{R}^{n}$, $S_{d}\succ\bm{0}$ are given. The mismatch between the desired distribution and the distribution of the actual terminal state $x_{N}$ is penalized as a terminal cost quantified using the squared 2-Wasserstein distance $W_{2}^{2}\left(\cdot,
\cdot\right)$ between those two distributions. We refer the readers to \cite[Sec. II]{p:balci2020covariance} for the details on problem formulation.

To recover the statistics of the terminal state $x_{N}$ from the concatenated state $\bs{x}$, the following relation will be useful:
\begin{equation}
    x_{N} = \bs{F}\bs{x}, \quad \bs{F}:=[\bs{0}, \dots, \bs{0}, I_{n_x}].
\end{equation}

It was shown in \cite{p:balci2020covariance} that the problem of discrete time covariance steering with Wasserstein terminal cost subject to \eqref{eq:system-dynamics} (or equivalently (\ref{s:simplified-dynamics})), can be reduced to a difference of convex functions program, provided the control policy is parameterized as
\begin{equation}\label{eq:statefeedbackpolicy}
    u_k = u_{\mathrm{ff}, k} + \sum_{t=0}^{k} K_{(k,t)} (x_t - \Bar{x}_t)
\end{equation}
where $\Bar{x}_t:=\mathbb{E}\left[x_t\right]$, and the parameters of the control policy are $u_{\mathrm{ff}, k} \in \mathbb{R}^{n_x}$, $K_{(k,t)} \in \mathbb{R}^{n_u \times n_x}$ for all $\{(k, t) \in \mathbb{N}_0\mid k \geq t\}$.
The concatenated control input $\bs{u}$ can be written as
\begin{equation}
    \bm{u} := \bm{u}_{\mathrm{ff}} + \bm{K} (\bm{x} - \bm{\Bar{x}})
\end{equation}
where $\bm{K} := \left[ \tilde{\bm{K}} ~~ \bm{0} \right]$, and
\begin{align}
    \Tilde{\bm{K}} :=
    \begin{bmatrix}
    K_{(0,0)} & \bm{0} & \dots & \bm{0}  \\
    K_{(1,0)} & K_{(1,1)} & \dots & \bm{0}  \\
    \vdots & \vdots & \ddots & \vdots  \\
    K_{(N-1, 0)} & K_{(N-1, 1)} & \dots & K_{(N-1, N-1)} \\
    \end{bmatrix}.
\end{align}
The controller synthesis thus amounts to computing the optimal feedforward control and feedback gain pair $(\uff,\bm{K})$.

In \cite{p:balci2020covariance}, the authors proposed a bijective mapping $\bm{K}\mapsto \bm{\Theta}$ and back, given by
\begin{subequations}\label{transformations}
\begin{gather}
    \bs{\Theta} := \bs{K}(I - \bs{H_u K})^{-1},\label{eqx}\\
    \bs{K} := (I + \bs{\Theta H_u})^{-1} \bs{\Theta}. \label{eqinvtrans}
\end{gather}
We have
\begin{align}
    (I - \bs{H_u K})^{-1} & = I + \bs{H_u K} (I - \bs{H_u K})^{-1} \nonumber \\ 
     & = (I + \bs{H_u \Theta}).
\end{align}
\end{subequations}
With the new feedback gain parameterization $\bm{\Theta}$, it was 
deduced in \cite{p:balci2020covariance} that the optimal pair $(\uff, \bs{\Theta})$ minimizes the objective $J : \mathbb{R}^{Nn_{u}}\times\mathbb{R}^{N n_u \times (N+1) n_x}\mapsto \mathbb{R}_{\geq 0}$, given by
\begin{align}\label{eq:total-objective}
    J(\uff, \bs{\Theta}) = J^{\text{cost-to-go}}(\uff, \bs{\Theta}) + \lambda W_2^2 (\uff, \bs{\Theta})
\end{align}
where $\lambda>0$ is given, and
\begin{align}\label{eq:control-objective}
    J^{\text{cost-to-go}}(\uff, \bs{\Theta}) = {\rm{trace}}\left(\bs{\Theta} \Tilde{S}  \bs{\Theta}^{\mathrm{T}}\right) + \lVert \uff \rVert_2^{2},
\end{align}
and
\begin{align}\label{eq:wasserstein-objective}
    & W_2^2(\uff, \bs{\Theta}) = \lVert \mu_d - (\bs{\Gamma}\mu_0 + \bs{H_u} \uff)  \rVert_{2}^{2} \nonumber \\
    & ~ + {\rm{trace}}\!\left(\!\bs{F}(I+\bs{H_u \Theta}) \Tilde{S} (I+\bs{H_u \Theta})^{\mathrm{T}} \bs{F}^{\mathrm{T}} + S_d\!\right) \nonumber \\
    & ~ -2 {\rm{trace}}\!\left(\!(\sqrt{S_d} \bs{F}(I+\bs{H_u \Theta}) \Tilde{S} (I+\bs{H_u \Theta})^{\mathrm{T}} \bs{F}^{\mathrm{T}} \sqrt{S_d})^{\frac{1}{2}}\!\right),
\end{align}
where 
\begin{equation}
    \Tilde{S} := \bs{\Gamma}S_0\bs{\Gamma}\t + \bs{H_w} \mathbf{W} \bs{H_w}\t
\label{deftildeS}    
\end{equation}
and the block diagonal matrix $\mathbf{W} := \operatorname{blkdiag}(S_w,  \dots, S_w)$.

\begin{proposition}\label{prop:tildeSposdef}
Consider $\Tilde{S}$ as in \eqref{deftildeS}. Then $\Tilde{S}\succ\bm{0}$.
\end{proposition}
\begin{proof}
From \eqref{deftildeS}, it is clear that $\Tilde{S}\succeq\bm{0}$. Suppose if possible that $\Tilde{S}$ is singular. Then there exists $(N+1)n_x\times 1$ vector $v\neq 0$ such that $v\t\Tilde{S}v=v\t(\bs{\Gamma}S_0\bs{\Gamma}\t + \bs{H_w} \mathbf{W} \bs{H_w}\t)v=0$, which in turn, is possible iff $\Gamma\t v = 0$ and $\bs{H_w}\t v=0$, since $S_0\succ\bm{0}$, $S_w\succ\bm{0}$.

Now let $v:=\left[v_0\t, v_1\t, \hdots, v_{N}\t\right]\t$ where the sub-vector $v_{i}\in\mathbb{R}^{n_x}$ for all $i\in\mathbb{N}_{0}[N]$. From $\bs{H_w}\t v=0$, we get $v_{1}=v_{2}=\hdots=v_{N}=0$ since the matrices $\{G_{k}\}_{k=0}^{N-1}$ are full rank per our assumption. In $\Gamma\t v = 0$, substituting $v_{1}=v_{2}=\hdots=v_{N}=0$, yields $v_{0}=0$. Thus, $v=0$ which contradicts our hypothesis. Therefore, the positive semidefinite matrix $\Tilde{S}$ is nonsingular, i.e., $\Tilde{S}\succ\bm{0}$.
\end{proof}

\begin{remark}\label{remark:theta}
An important consideration is that in order to ensure the causality of the control policy, the matrix $\bs{\Theta} \in \mathbb{R}^{N n_u \times (N+1) n_x}$ should be constrained to be block lower triangular of the form
\begin{equation}\label{eq:thetadef}
    \bs{\Theta} := 
    \begin{bmatrix}
    \theta_{0,0} & \bs{0} & \dots & \bs{0} & \bs{0} \\
    \theta_{1,0} & \theta_{1,1} & \dots & \bs{0} & \bs{0} \\
    \vdots & \vdots & \vdots & \vdots & \vdots \\
    \theta_{N-1, 0}& \theta_{N-1, 1} & \dots & \theta_{N-1,N-1} & \bs{0}
    \end{bmatrix}
\end{equation}
where $\theta_{i,j} \in \mathbb{R}^{n_u\times n_x}$ for all index pairs $(i, j)$.
\end{remark}

The block lower triangular condition on $\bs{\Theta}$ in Remark \ref{remark:theta} can be equivalently expressed as
\begin{equation}\label{CausalityFirstForm}
    \theta_{i, j} = \bm{0}\;\forall\;(i,j) \in \mathbb{N}_{0}[N-1]\times\mathbb{N}_{0}[N]\;\text{such that}\; j > i.
\end{equation}
We transcribe this constraint in terms of the decision variable $\bm{\Theta}$ as
\begin{equation}\label{CausalityConstr}   
    \mathcal{E}_{u,i} \bs{\Theta} \mathcal{E}_{x,j}^{\mathrm{T}} = \bs{0}\;\forall\;(i,j) \in \mathbb{N}_{0}[N-1]\times\mathbb{N}_{0}[N]\;\text{such that}\; j > i,
\end{equation}
where $\mathcal{E}_{u,i} \in \mathbb{R}^{n_u \times n_u N}$ and $\mathcal{E}_{x,i} \in \mathbb{R}^{n_x \times n_x (N+1)}$ are defined as block vectors whose $(i+1)$\textsuperscript{th} and $(j+1)$\textsuperscript{th}  blocks are equal to the identity matrices of suitable dimensions; all the other blocks are equal to the zero matrix.
For example,
\begin{align*}
    \mathcal{E}_{u,0} = \left[ I_{n_u}, \bm{0}, \dots, \bm{0} \right], \quad \mathcal{E}_{x,1} = \left[ \bm{0}, I_{n_x}, \bm{0}, \dots, \bm{0} \right],
\end{align*}
where $I_{\nu}$ denotes an identity matrix of size $\nu\times\nu$.

It is clear that \eqref{eq:control-objective} is a convex quadratic function in its arguments with Lipschitz continuous gradient. The squared Wasserstein distance \eqref{eq:wasserstein-objective} is a difference of convex functions in $(\uff, \bs{\Theta})$, and it can be shown that it is also Lipschitz continuous gradient. Thus, the objective $J$ in \eqref{eq:total-objective} is a difference of convex functions in the decision variables, and as such, it is unclear when it might in fact be convex. 
In \cite{p:balci2020covariance}, we used convex-concave procedure \cite{p:yuille2003concave} to numerically compute the optimal solution. 
In our numerical experiments, we observed multiple local minima which motivates investigating the conditions of optimality for \eqref{eq:total-objective}. This is what we pursue in Sections \ref{sec:FirstOrderConditions} and \ref{sec:SecondOrderConditions}. Before doing so, we show that the objective $J$ in \eqref{eq:total-objective} is not convex in general but there exists a global minimizer.


\begin{proposition}\label{prop:coercive}
The problem of minimizing the objective $J$ in \eqref{eq:total-objective} subject to the constraints \eqref{CausalityConstr}, admits a global minimizing pair $(\uff,\bm{\Theta})$.
\end{proposition}
\begin{proof}
The objective $J$ in \eqref{eq:total-objective} is continuous and coercive (i.e., $\lim_{\|\uff\|_{2}\rightarrow\infty,\|\bm{\Theta}\|_{2}\rightarrow\infty} J = \infty$) in its arguments. 

That $J$ is continuous in $(\uff, \bs{\Theta})$ is immediate. To establish coercivity, following \cite[see equation (26)]{p:balci2020covariance}, we write
\begin{align}
J\left(\uff,\bs{\Theta}\right) = J_{1}\left(\uff\right) + J_{2}\left(\bs{\Theta}\right) + J_{3}\left(\bs{\Theta}\right) - J_{4}\left(\bs{\Theta}\right),
\label{JasJ1J2J3J4} 
\end{align}
where
\begin{subequations}
\begin{align}
&J_{1}\left(\uff\right) := \|\uff\|_{2}^{2} + \lambda \|\bs{F}\left(\bs{\Gamma}\mu_{0} + \bs{H_u}\uff\right) - \mu_{d}\|_{2}^{2}, \label{J1}\\
&J_{2}\left(\bs{\Theta}\right) := {\rm{trace}}\left(\bm{\Theta}\tilde{S}\bm{\Theta}\t\right), \label{J2}\\
&J_{3}\left(\bs{\Theta}\right) := \lambda\:{\rm{trace}}\left(\bm{F}\left(I+\bs{H_{u}\Theta}\right)\tilde{S}\left(I+\bs{H_{u}\Theta}\right)\t\bm{F}\t \right.\nonumber\\
&\qquad\qquad\qquad\qquad\left.+ S_{d}\right), \label{J3}\\
&J_{4}\left(\bs{\Theta}\right) := 2\lambda\:{\rm{trace}}\left(\left(S_{d}^{\frac{1}{2}} \bm{F}\left(I+\bs{H_{u}\Theta}\right)\tilde{S}\left(I+\bs{H_{u}\Theta}\right)\t\right.\right.\nonumber\\
&\qquad\qquad\qquad\qquad\quad\left.\left.\bm{F}\t S_{d}^{\frac{1}{2}}\right)^{\frac{1}{2}}\right).\label{J4}
\end{align}
\label{J1J2J3J4}
\end{subequations}
Since $J_1(\uff)$ in \eqref{J1} is strictly convex quadratic in $\uff$, it is clear that $J_1(\uff) \rightarrow \infty$ as $\| \uff \|_{2} \rightarrow \infty$. 

We note that $J_2(\bm{\Theta})$ equals $\tr{\bm{\Theta}\t \bm{\Theta} \Tilde{S}}$ due to invariance of the trace operator under cyclic permutation. Using \eqref{kronANDvec} and \eqref{FrobInnerProductANDVec}, we then write
\begin{align}\label{J2vecform}
    J_2(\bm{\Theta}) = \Vec{(\bm{\Theta})}\t (\Tilde{S} \otimes I) \Vec{(\bm{\Theta})}.
\end{align}
Since $I \succ \bm{0}, \Tilde{S} \succ \bm{0}$ (by Proposition \ref{prop:tildeSposdef}), we have $\Tilde{S} \otimes I \succ \bm{0}$. Thus, $J_2(\bm{\Theta})$ is a strictly convex quadratic function and $J_2(\bm{\Theta}) \rightarrow \infty$ as $\|\bm{\Theta}\|_{2} \rightarrow \infty$. 

Finally, since $J_3(\bm{\Theta}) - J_4(\bm{\Theta})$ comes from the expression of the squared Wasserstein distance which is lower bounded by zero, the function $J_3(\bm{\Theta}) - J_4(\bm{\Theta}) \geq 0$ for all $\bm{\Theta}$. Thus, $\lim_{\|\uff\|_{2}\rightarrow\infty,\|\bm{\Theta}\|_{2}\rightarrow\infty} J = \infty$, i.e., the function $J(\uff, \bm{\Theta})$ in \eqref{JasJ1J2J3J4} is coercive.   

Moreover, the constraint set 
\[\{\left(\uff,\bm{\Theta}\right)\in\mathbb{R}^{N n_u}\times \mathbb{R}^{N n_u\times (N+1)n_x} \mid \eqref{CausalityConstr}\;\text{holds}\}\] 
is closed. Thus, minimizing the objective $J$ in \eqref{eq:total-objective} subject to the constraints \eqref{CausalityConstr}, amounts to minimizing a continuous coercive function over a closed set. Hence, there exists global minimizing pair $(\uff,\bm{\Theta})$ for this problem.
\end{proof}
Notice that Proposition \ref{prop:coercive} only guarantees the existence of global minimizer; it does not guarantee uniqueness. The following example shows that in general, $J$ is nonconvex, and there might be multiple local minima which makes it challenging to find the global minimizer. 


\begin{figure}[t]
    \centering
    \includegraphics[width=\linewidth]{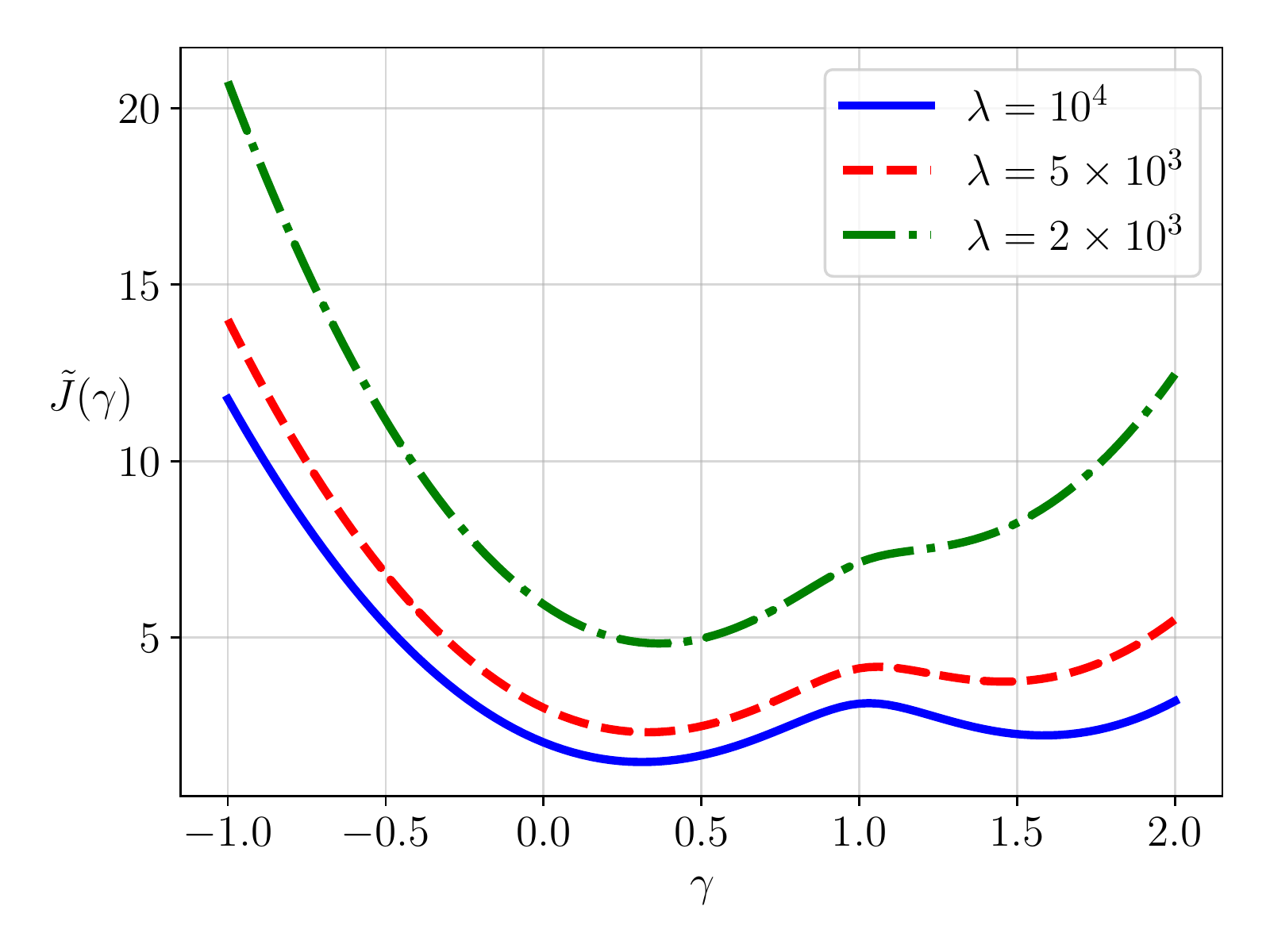}
    \vspace*{-0.15in}
    \caption{$\tilde{J}(\gamma)$ versus $\gamma$ for Example \ref{example:MultipleLocalMinima}.}
    \label{fig:nonconvex}
\end{figure}

\begin{example}(\textbf{Nonconvexity of $J$})\label{example:MultipleLocalMinima}
Consider the system matrices
\begin{align*}
    A_k = 
    \begin{bmatrix}
    1.0 & 0.1 \\
    0.0 & 1.0
    \end{bmatrix},
    B_k =
    \begin{bmatrix}
    0.0 \\
    0.1
    \end{bmatrix},
    G_k = 
    \begin{bmatrix}
    1.0 & 0.0 \\
    0.0 & 1.0 
    \end{bmatrix}\:\forall k\in\mathbb{N}_{0},
\end{align*}
with time horizon $N=10$. The initial and desired mean vectors are $\mu_0 = [0.0, 0.0]\t$, $\mu_{d}=[10.0, 5.0]\t$, respectively. The initial covariance is $S_0 = I_2$.
With this data, for two different desired distributions $\mathcal{N}\left(\mu_{d},S_{d1}\right)$ and $\mathcal{N}\left(\mu_{d},S_{d2}\right)$ with
\begin{align*}
    S_{d1} =
    \begin{bmatrix}
    4.0 & -2.0 \\ -2.0 & 2.0
    \end{bmatrix},\quad 
    S_{d2} =
    \begin{bmatrix}
    0.2 & 0.0 \\ 0.0 & 0.1
    \end{bmatrix},
\end{align*}
we numerically computed (using convex-concave procedure, see \cite[Sec. IV]{p:balci2020covariance}) the minimizers ($\uff^{1}, \bm{\Theta}^{1}$) and ($\uff^{1}, \bm{\Theta}^{2}$).
Since the desired mean vector is the same $\mu_{d}$ in both cases, $\bm{u}_{\rm{ff}}^{1} = \bm{u}_{\rm{ff}}^{2}$.

For $\gamma\in\mathbb{R}$, define an affine function $g(\gamma) := (\bm{u}_{\rm{ff}}^{1} + \gamma (\bm{u}_{\rm{ff}}^{2}-\bm{u}_{\rm{ff}}^{1}),~\bm{\Theta}^{1} + \gamma (\bm{\Theta}^{2}-\bm{\Theta}^{1}) )$ and let $\tilde{J}(\gamma):= J \circ g(\gamma)$. The function $J$ is convex iff its restriction to a line, i.e., $\tilde{J}$ is convex. Fig. \ref{fig:nonconvex} shows that the function $\tilde{J}(\gamma)$ has multiple local minima, thus the function $J(\bm{u}_{\rm{ff}}, \bm{\Theta})$ is nonconvex.
\end{example}

\section{First Order Conditions for Optimality}\label{sec:FirstOrderConditions}
Recall the function $J(\uff, \bm{\Theta})$ in \eqref{JasJ1J2J3J4} and \eqref{J1J2J3J4}. We define the index set
\[\mathcal{I}:=\{(i,j)\in\mathbb{N}_{0}[N-1]\times\mathbb{N}_{0}[N]\mid j>i\}.\] 
Now consider the Lagrangian 
\begin{align}
\mathcal{L}\left(\uff, \bs{\Theta},\bs{N}\right) = J\left(\uff,\bs{\Theta}\right) + \!\!\displaystyle\sum_{(i,j)\in\mathcal{I}}\!\langle\bm{\Psi}_{i,j}, \mathcal{E}_{u,i} \bs{\Theta} \mathcal{E}_{x,j}^{\mathrm{T}}\rangle    
\label{Lagrangian}    
\end{align}
where $\bs{\Psi}_{i,j}$ is the Lagrange multiplier matrix associated with the $(i,j)$\textsuperscript{th} linear equality constraint \eqref{CausalityConstr} for all $(i,j)\in\mathcal{I}$, and $\langle\cdot,\cdot\rangle$ denotes the Frobenius inner product. Let us denote the optimal pair as $\left(\uff^{\star},\bs{\Theta}^{\star}\right)$. The first order necessary conditions for optimality are
\[\dfrac{\partial \mathcal{L}}{\partial\bm{u}_{{\rm{ff}}}}\bigg\vert_{\left(\uff^{\star},\bs{\Theta}^{\star}\right)} = \bm{0}, \quad \dfrac{\partial \mathcal{L}}{\partial\bs{\Theta}}\bigg\vert_{\left(\uff^{\star},\bs{\Theta}^{\star}\right)} = \bm{0},\quad \mathcal{E}_{u,i} \bs{\Theta}^{\star} \mathcal{E}_{x,j}^{\mathrm{T}} = \bs{0}.\]
We next compute the gradients of $\mathcal{L}$ w.r.t. the vector variable $\uff$ and the matrix variable $\bs{\Theta}$, respectively, and use them to determine the pair $\left(\uff^{\star},\bs{\Theta}^{\star}\right)$. 

\subsection{The optimal feedforward control}
From \eqref{JasJ1J2J3J4}, \eqref{J1} and \eqref{Lagrangian}, we obtain
\begin{align*}
\dfrac{\partial \mathcal{L}}{\partial\bm{u}_{{\rm{ff}}}} \!= \!\dfrac{\partial J_{1}}{\partial\bm{u}_{{\rm{ff}}}} \!=\! 2\uff + 2\lambda\bs{H_{u}}\t \bs{F}\t\!\!\left(\bs{F}\left(\bs{\Gamma}\mu_{0} + \bs{H_u}\uff\right) - \mu_{d}\right),    
\end{align*}
and therefore, the condition $\dfrac{\partial \mathcal{L}}{\partial\bm{u}_{{\rm{ff}}}}\bigg\vert_{\left(\uff^{\star},\bs{\Theta}^{\star}\right)} = \bm{0}$ yields a linear matrix-vector equation
\begin{align}
\left(I + \lambda \bs{H_u}\t\bs{F}\t\bs{FH_u}\right) \uff^{\star} = \lambda \bs{H_u}\t\bs{F}\t\left(\mu_{d} - \bs{F\Gamma}\mu_{0}\right). 
\label{uffMatrixVectorEqn}    
\end{align}
The matrix $I + \lambda \bs{H_u}\t\bs{F}\t\bs{FH_u}$ in (\ref{uffMatrixVectorEqn}) is positive definite (thus non-singular), and the optimal feedforward control is 
\begin{align}
&\uff^{\star} = \left(I + \lambda \bs{H_u}\t\bs{F}\t\bs{FH_u}\right)^{-1} \lambda \bs{H_u}\t\bs{F}\t\left(\mu_{d} - \bs{F\Gamma}\mu_{0}\right)\nonumber\\
&= \left(I - \lambda \bs{H_u}\t\bs{F}\t\left(I + \lambda\bs{FH_u}\bs{H_u}\t\bs{F}\t\right)\bs{FH_u}\right)\lambda\bs{H_u}\t\bs{F}\t\nonumber\\
&\qquad\qquad\qquad\qquad\qquad\qquad\qquad\left(\mu_{d} - \bs{F\Gamma}\mu_{0}\right).
\label{uffstar}
\end{align}
Notice that $\uff^{\star}$ is unique.

\subsection{The optimal feedback gain}
From (\ref{JasJ1J2J3J4}), (\ref{J1J2J3J4}) and (\ref{Lagrangian}), we have
\begin{align}
\dfrac{\partial \mathcal{L}}{\partial\bs{\Theta}} &\!=\! \dfrac{\partial J_{2}}{\partial\bs{\Theta}} + \dfrac{\partial J_{3}}{\partial\bs{\Theta}} - \dfrac{\partial J_{4}}{\partial\bs{\Theta}} +\!\!\! \displaystyle\sum_{(i,j)\in\mathcal{I}}\!\dfrac{\partial}{\partial\bm{\Theta}}{\rm{trace}}\left(\mathcal{E}_{x,j}\t\bm{\Psi}_{i,j}\t\mathcal{E}_{u,i}\bm{\Theta}\right)\nonumber\\
&\!=\! \dfrac{\partial J_{2}}{\partial\bs{\Theta}} + \dfrac{\partial J_{3}}{\partial\bs{\Theta}} - \dfrac{\partial J_{4}}{\partial\bs{\Theta}} + \!\!\displaystyle\sum_{(i,j)\in\mathcal{I}}\mathcal{E}_{u,i}\t\bm{\Psi}_{i,j}\mathcal{E}_{x,j}.
\label{PartialLPartialTheta}
\end{align}
Notice that 
\begin{align}
\dfrac{\partial J_{2}}{\partial\bs{\Theta}} = \!\dfrac{\partial}{\partial\bs{\Theta}} {\rm{trace}}\!\left(\bm{\Theta}\tilde{S}\bm{\Theta}\t\right) \!= \!\dfrac{\partial}{\partial\bs{\Theta}} {\rm{trace}}\!\left(\bm{\Theta}\t\bm{\Theta}\tilde{S}\right) \!=\! 2\bs{\Theta}\tilde{S},
\label{PartialJ2PartialTheta}    
\end{align}
which follows from the invariance of trace under cyclic permutation, and the from the fact that the directional derivative (in the matricial direction $\bm{Z}$)
\begin{align*}
&\displaystyle\lim_{h\rightarrow 0} \frac{1}{h}\!\left[{\rm{trace}}\!\left(\left(\bm{\Theta}+h\bm{Z}\right)\t\left(\bm{\Theta}+h\bm{Z}\right)\tilde{S}\right)\! - {\rm{trace}}\!\left(\bm{\Theta}\t\bm{\Theta}\tilde{S}\right)\!\right]\\
&= \langle 2\bm{\Theta}\tilde{S},\bm{Z}\rangle.
\end{align*}
Furthermore, let $J_{31}(\bm{\Theta}):=\bm{\Theta}\bm{\Theta}\t$, $J_{32}(\bm{\Theta}):=\bm{FH_u\Theta}\tilde{S}^{\frac{1}{2}}$, $\bm{Y}_{3} := \dfrac{\partial}{\partial\bm{\Theta}}{\rm{trace}}\left(J_{31}\circ J_{32}(\bm{\Theta})\right)$, and notice that
\begin{align}
\dfrac{\partial J_{3}}{\partial\bs{\Theta}} \!=\! \lambda\!\left(\!2\bs{H_u}\t\bs{F}\t\bs{F}\tilde{S} + \bm{Y}_{3}\right).
\label{PartialJ3PartialTheta}    
\end{align}
From the chain rule of Jacobians, we have
\begin{align}
& \differential\:{\rm{trace}}\left(J_{31}\circ J_{32}(\bm{\Theta})\right)\nonumber\\
=&\left(\vec\left(I\right)\right)\t \jac J_{31}\left(J_{32}(\bm{\Theta})\right) \jac J_{32}(\bm{\Theta}) \:\differential\:\vec(\bm{\Theta}) \nonumber\\
= &\left(\vec(\bm{Y}_{3})\right)\t \differential\:\vec(\bm{\Theta}),
\label{Y3JacobianRelation}    
\end{align}
wherein using Lemma \ref{JacOfAXB} and \ref{JacOfXXt} from Appendix, we get
\begin{subequations}
\begin{align}
\jac J_{32}(\bm{\Theta}) &= \tilde{S}^{\frac{1}{2}} \otimes \bs{F H_u}, \label{JacJ32}\\
\jac J_{31 }(\bm{\Theta}) &= \left(I + K_{0}\right)\left(\bm{\Theta}\otimes I\right) \label{JacJ31}.
\end{align}
\label{JacJ32J31}
\end{subequations}
Combining \eqref{Y3JacobianRelation} and \eqref{JacJ32J31}, we obtain
\begin{subequations}
\begin{align}
&\vec\left(\bm{Y}_{3}\right) = \left(\jac J_{32}(\bm{\Theta})\right)\t \left(\jac J_{31}\left(J_{32}(\bm{\Theta})\right)\right)\t \vec\left(I\right)\nonumber\\
&= \left(\tilde{S}^{\frac{1}{2}}\otimes \bs{H_u}\t\bs{F}\t\right)\left(\tilde{S}^{\frac{1}{2}}\bs{\Theta}\t\bs{H_u}\t\bs{F}\t\otimes I\right)\left(I + K_{0}\right)\vec\left(I\right)\nonumber\\
&= 2\left(\tilde{S}\bs{\Theta}\t\bs{H_u}\t\bs{F}\t\otimes\bs{H_u}\t\bs{F}\t\right) \vec\left(I\right) \label{IplusKtimesvecI}\\
&= 2\:\vec\left(\bs{H_u}\t\bs{F}\t\bs{F}\bs{H_u}\bs{\Theta}\tilde{S}\right),
\label{IdentifyY3}
\end{align}
\end{subequations}
wherein \eqref{IplusKtimesvecI} used $\left(I + K_{0}\right)\vec\left(I\right) = 2\:\vec\left(I\right)$, and \eqref{IdentifyY3} follows from \eqref{kronANDvec}. 

From \eqref{IdentifyY3}, we identify $\bm{Y}_{3} = 2\bs{H_u}\t\bs{F}\t\bs{F}\bs{H_u}\bs{\Theta}\tilde{S}$, which together with \eqref{PartialJ3PartialTheta}, yields
\begin{align}
\dfrac{\partial J_{3}}{\partial\bs{\Theta}} \!=\!  2\lambda\bs{H_u}\t\bs{F}\t\bs{F}\left(I + \bs{H_u \Theta}\right)\tilde{S}.
\label{PartialJ3PartialThetaFinal}    
\end{align}

Next, let 
\begin{align*}
J_{41}\left(\bm{\Theta}\right) &:=\left(\sqrt{S}_{d} \bm{\Theta}\sqrt{S}_{d}\right)^{\frac{1}{2}},\\
J_{42}\left(\bm{\Theta}\right) &:= \bm{\Theta\Theta}\t,\quad
J_{43}\left(\bm{\Theta}\right) := \bm{F}\left(I+\bs{H_u\Theta}\right)\tilde{S}^{\frac{1}{2}},\\
\bm{Y}_{4} &:= \dfrac{\partial}{\partial\bm{\Theta}}{\rm{trace}}\left(J_{41}\circ J_{42} \circ J_{43}(\bm{\Theta})\right),
\end{align*}
and notice that $\dfrac{\partial J_{4}}{\partial\bm{\Theta}} = 2\lambda \bm{Y}_{4}$. Therefore, writing
\begin{align*}
&\differential\:{\rm{trace}}\left(J_{41}\circ J_{42}\circ J_{43}(\bm{\Theta})\right)\\
=&\left(\vec\left(I\right)\right)\t \jac J_{41}\left(J_{42}\left(J_{43}(\bm{\Theta})\right)\right) \jac J_{42}(J_{43}\left(\bm{\Theta}\right))\\
&\qquad\qquad\qquad\qquad\qquad\qquad \jac J_{43}\left(\bm{\Theta}\right) \:\differential\:\vec(\bm{\Theta})\\
=& \left(\vec(\bm{Y}_{4})\right)\t \differential\:\vec(\bm{\Theta}),
\end{align*}
we obtain
\begin{align}
&\vec\left(\bm{Y}_{4}\right) = \left(\jac J_{43}(\bm{\Theta})\right)\t \left(\jac J_{42}\left(J_{43}(\bm{\Theta})\right)\right)\t \nonumber\\
&\qquad\qquad\qquad\qquad\left(\jac J_{41}\left(J_{42}\left(J_{43}(\bm{\Theta})\right)\right)\right)\t\vec\left(I\right).
\label{vecY4}
\end{align}
To proceed further, we utilize the following results:
\begin{subequations}
\begin{align}
\jac J_{43}(\bm{\Theta}) &= \jac\!\left(\!\bs{F}\tilde{S}^{\frac{1}{2}}+\bs{F H_u \Theta}\tilde{S}^{\frac{1}{2}}\!\right) \!= \tilde{S}^{\frac{1}{2}}\otimes\bs{F H_u}, \label{JacJ43}\\
\jac J_{42}(\bm{\Theta}) &= \left(I + K_{0}\right)\left(\bs{\Theta}\otimes I\right), \label{JacJ42}\\
\jac J_{41}(\bm{\Theta}) &= \!\left(\!\left(S_{d}^{\frac{1}{2}}\bm{\Theta}S_{d}^{\frac{1}{2}}\right)^{\!\frac{1}{2}} \!\oplus\! \left(S_{d}^{\frac{1}{2}}\bm{\Theta}S_{d}^{\frac{1}{2}}\right)^{\!\frac{1}{2}}\!\right)^{\!\!-1}\!\!\left(S_{d}^{\frac{1}{2}}\otimes S_{d}^{\frac{1}{2}}\right). \label{JacJ41}
\end{align}
\label{JacJ43J42J41}
\end{subequations}
The result (\ref{JacJ43}) follows from Lemma \ref{JacOfAXB} while (\ref{JacJ42}) follows from Lemma \ref{JacOfXXt}. The expression (\ref{JacJ41}) follows from \cite[equation (30)]{halder2016finite}. 

Now let
\begin{align}
\Omega := \bm{F}\left(I + \bs{H_u \Theta}\right),    
\label{defOmega}    
\end{align}
which is a linear function of $\bs{\Theta}$. Substituting \eqref{JacJ43J42J41} in \eqref{vecY4}, and then using \eqref{ProductOfKronProduct}, \eqref{IplusK0commutesWithSelfKronProduct}, \eqref{IplusK0commutesWithSelfKronSumInverse}, and recalling $\left(I + K_{0}\right)\vec(I) = 2\:\vec(I)$, we obtain
\begin{align}
&\vec\left(\bm{Y}_{4}\right) = 2\left(\tilde{S}\Omega\t\bs{F}\t S_{d}^{\frac{1}{2}} \otimes \bs{H_u}\t\bs{F}\t S_{d}^{\frac{1}{2}}\right)\nonumber\\
& \left(\left(S_{d}^{\frac{1}{2}}\Omega\tilde{S}\Omega\t S_{d}^{\frac{1}{2}}\right)^{\frac{1}{2}} \oplus \left(S_{d}^{\frac{1}{2}}\Omega\tilde{S}\Omega\t S_{d}^{\frac{1}{2}}\right)^{\frac{1}{2}}\right)^{\!\!-1}\!\!\vec(I).
\end{align}
Therefore, following similar steps as in \cite[Appendix B, equations (32)-(35)]{halder2016finite}, we arrive at\footnote{The matrix $\Omega\tilde{S}\Omega\t$ is the right bottom corner block of size $n_x\times n_x$ from the $(N+1)n_x\times(N+1)n_x$ symmetric positive definite matrix $(I+\bs{H_u\Theta})\tilde{S}(I+\bs{H_u\Theta})\t$, and is thus symmetric positive definite.}
\begin{align}
\bm{Y}_{4} = \bs{H_u}\t\bs{F}\t \left(S_{d}\: \# \left(\Omega\tilde{S}\Omega\t\right)^{\!\!-1}\right) \Omega \tilde{S},    
\label{Y4}    
\end{align}
where $\#$ denotes the matrix geometric mean as in \eqref{matrixGM}. Hence
\begin{align}
\dfrac{\partial J_{4}}{\partial\bm{\Theta}} = 2\lambda \bm{Y}_{4} = 2\lambda \bs{H_u}\t\bs{F}\t \left(S_{d}\: \# \left(\Omega\tilde{S}\Omega\t\right)^{\!\!-1}\right) \Omega \tilde{S}.
\label{PartialJ4PartialThetaFinal}    
\end{align}

\begin{figure*}
\hrule
\begin{align}
&2\bm{\Theta}^{\star}\tilde{S} + 2\lambda\bs{H_u}\t\bs{F}\t\bs{F}\left(I + \bs{H_u \Theta^{\star}}\right)\tilde{S}\nonumber\\
&-2\lambda\bs{H_u}\t\bs{F}\t S_{d}^{\frac{1}{2}}\!\left(S_{d}^{-\frac{1}{2}}\left(\bm{F}\left(I + \bs{H_u} \bs{\Theta}^{\star}\right)\tilde{S}\left(I + \bs{\Theta}^{\star}{\t}\bs{H_u}\t\right)\bm{F}\t\right)^{\!-1} S_{d}^{-\frac{1}{2}}\right)^{\!\frac{1}{2}}\!S_{d}^{\frac{1}{2}}\bm{F}\left(I + \bs{H_u} \bs{\Theta}^{\star}\right) \tilde{S} + \!\!\!\displaystyle\sum_{(i,j)\in\mathcal{I}}\!\mathcal{E}_{u,i}\t\bm{\Psi}_{i,j}\mathcal{E}_{x,j} = \bm{0}.
\label{OptimalFeedback}
\end{align}
\hrule
\end{figure*}
Combining \eqref{PartialLPartialTheta}, \eqref{PartialJ2PartialTheta}, \eqref{PartialJ3PartialThetaFinal} and \eqref{PartialJ4PartialThetaFinal}, with $\dfrac{\partial \mathcal{L}}{\partial\bs{\Theta}}\bigg\vert_{\left(\uff^{\star},\bs{\Theta}^{\star}\right)} = \bm{0}$, we arrive at a nonlinear matrix equation in $\bm{\Theta}^{\star}$ given by \eqref{OptimalFeedback}. Thus, the primal feasibility \eqref{CausalityConstr} and the Lagrangian gradient \eqref{OptimalFeedback} together give the first order optimality conditions for $\bs{\Theta}^{\star}$.

Investigating the existence and uniqueness of solutions for the system \eqref{CausalityConstr} and \eqref{OptimalFeedback} appears technically challenging. Instead, we next focus on deriving the second order conditions for optimality. Specifically, we will derive an exact formula for the Hessian $\hess(J)$, and then use the same to deduce a sufficient condition for $\hess(J)$ to be positive definite, guaranteeing the uniqueness of $\bm{\Theta}^{\star}$.


\section{Second Order Conditions}\label{sec:SecondOrderConditions}
We start by noting that
\begin{align}
\hess\left(J\right) = \hess\left(J_{2}\right) + \hess\left(J_{3}\right) - \hess\left(J_{4}\right),
\label{HessianOfJ}    
\end{align}
where all Hessians are w.r.t. $\bm{\Theta}$. Applying Lemma \ref{JacOfAXB} on \eqref{PartialJ2PartialTheta} yields
\begin{align}
\hess\left(J_{2}\right) = 2\left(\tilde{S}\otimes I\right).
\label{HessOfJ2}    
\end{align}
Likewise, applying Lemma \ref{JacOfAXB} on \eqref{PartialJ3PartialThetaFinal} gives
\begin{align}
\hess\left(J_{3}\right) = 2\lambda\left(\tilde{S}\otimes \bs{H_u}\t\bs{F}\t\bs{F H_u}\right).
\label{HessOfJ3}    
\end{align}
To compute $\hess\left(J_{4}\right)$, we need to identify the Jacobian of \eqref{PartialJ4PartialThetaFinal} using \eqref{JacIdentificationRule}. We proceed by letting
\begin{align}
Q(\bm{\Theta}) &:= S_{d}^{\frac{1}{2}}\left(S_{d}^{-\frac{1}{2}}\left(\bm{\Theta}\tilde{S}\bm{\Theta}\t\right)^{\!-1} S_{d}^{-\frac{1}{2}}\right)^{\frac{1}{2}}S_{d}^{\frac{1}{2}}, \label{defQ}\\
P(\bm{\Theta}) &:= Q(\bm{\Theta})\bm{\Theta}.\label{defP}
\end{align}
Then $P\circ\Omega(\bm{\Theta}) = \left(S_{d}\: \# \left(\Omega\tilde{S}\Omega\t\right)^{\!\!-1}\right)\Omega$, and therefore
\begin{align}
&\:2\lambda\:\differential\:\vec\left(\bs{H_u}\t\bs{F}\t \left(S_{d}\: \# \left(\Omega\tilde{S}\Omega\t\right)^{\!\!-1}\right) \Omega \tilde{S}\right)\nonumber\\    
\stackrel{\eqref{kronANDvec}}{=}&\:\underbrace{2\lambda\left(\tilde{S}\otimes\bs{H_u}\t\bs{F}\t\right)\jac P\left(\Omega(\bm{\Theta})\right)\jac\Omega(\bm{\Theta})}_{\hess(J_{4})}\:\differential\:\vec\left(\bm{\Theta}\right). 
\label{IdentifyingHessJ4}
\end{align}

From \eqref{defP}, we have
\begin{align}
&\differential\:\vec\left(P\left(\bm{\Theta}\right)\right) = \vec\left((\differential Q)\bm{\Theta}\right) + \vec\left(Q(\differential\bm{\Theta})\right)\nonumber\\
&\stackrel{\eqref{kronANDvec}}{=} \left(\bm{\Theta}\t\otimes I\right)\vec(\differential Q) + \left(I\otimes Q\right)\vec(\differential\bm{\Theta})\nonumber\\
&\stackrel{\eqref{JacIdentificationRule}}{=} \big\{\left(\bm{\Theta}\t\otimes I\right)\jac Q + \left(I\otimes Q\right)\big\}\:\vec\left(\differential\bm{\Theta}\right)\nonumber,
\end{align}
and consequently,
\begin{align}
\jac P(\bm{\Theta}) = \left(\bm{\Theta}\t\otimes I\right)\jac Q + \left(I\otimes Q\right).
\label{jacPfromjacQ}    
\end{align}
Thus, our strategy is to follow the computational sequence:
\begin{align}
\jac Q \xrightarrow{\eqref{jacPfromjacQ}} \jac P \xrightarrow{\eqref{IdentifyingHessJ4}} \hess(J_{4}).    
\label{CompSequence}    
\end{align}
To compute $\jac Q$, let us define the following functions with symmetric positive definite matrix $X$ as the argument:
\begin{align*}
&Q_{1}(X) := X\tilde{S}X\t, \qquad\qquad\;\;\: Q_{2}(X) := X^{-1},\\
&Q_{3}(X) := \left(S_{d}^{-\frac{1}{2}}X S_{d}^{-\frac{1}{2}}\right)^{\frac{1}{2}}, \quad Q_{4}(X) := S_{d}^{\frac{1}{2}}XS_{d}^{\frac{1}{2}}.
\end{align*}
Then $Q(X) = Q_{4} \circ Q_{3} \circ Q_{2} \circ Q_{1} (X)$, and thus
\begin{align}
\jac Q(X) = &\jac Q_{4}\left(Q_{3}\left(Q_{2}\left(Q_{1}(X)\right)\right)\right)\:\jac Q_{3}\left(Q_{2}\left(Q_{1}(X)\right)\right)\nonumber\\
& \quad \jac Q_{2}\left(Q_{1}(X)\right)\: \jac Q_{1}(X),
\label{jacQ}
\end{align}
wherein
{\small{\begin{align*}
\jac Q_{1}(X)\! &= (I+K_{0})(I\otimes X\tilde{S}) \qquad\qquad\qquad\text{from Lemma \ref{JacOfXSXt},}\\
\jac Q_{2}(X)\! &= -\left(X^{-{\mathrm{T}}}\otimes X^{-1}\right) \qquad\qquad\qquad\;\text{from Lemma \ref{JacOfXinv},}\\
\jac Q_{3}(X)\! &=\!\!\left(\!\!\left(S_{d}^{-\frac{1}{2}}XS_{d}^{-\frac{1}{2}}\right)^{\!\frac{1}{2}}\!\!\oplus\! \left(S_{d}^{-\frac{1}{2}}XS_{d}^{-\frac{1}{2}}\right)^{\!\frac{1}{2}}\!\right)\!\!\left(\!S_{d}^{-\frac{1}{2}}\otimes S_{d}^{-\frac{1}{2}}\!\right)\\ 
&\qquad\qquad\qquad\qquad\qquad\qquad\quad\text{from \cite[equation (30)]{halder2016finite},}\\
\jac Q_{4}(X)\! &= S_{d}^{\frac{1}{2}}\otimes S_{d}^{\frac{1}{2}}  \qquad\qquad\qquad\qquad\qquad\text{from Lemma \ref{JacOfAXB}.}
\end{align*}}}
Substituting these Jacobians back in \eqref{jacQ} gives
\begin{align}
\jac Q(X) = \left(S_{d}^{\frac{1}{2}}\otimes S_{d}^{\frac{1}{2}}\right)\left(\!\!\left(S_{d}^{-\frac{1}{2}}(X\tilde{S}X\t)^{-1}S_{d}^{-\frac{1}{2}}\right)^{\!\frac{1}{2}}\!\!\oplus\right.\nonumber\\ 
\left.\left(S_{d}^{-\frac{1}{2}}(X\tilde{S}X\t)^{-1}S_{d}^{-\frac{1}{2}}\right)^{\!\frac{1}{2}}\!\right)
\left(\!S_{d}^{-\frac{1}{2}}\otimes S_{d}^{-\frac{1}{2}}\!\right)\nonumber\\
\left(-(X\tilde{S}X\t)^{-1}\otimes(X\tilde{S}X\t)^{-1}\right)(I+K_{0})(X\tilde{S}\otimes I).
\label{jacQfinal}    
\end{align}
Having \eqref{jacQfinal}, we follow the computational sequence \eqref{CompSequence} to obtain the following result.
\begin{theorem}\label{thm:Hessian}(\textbf{Hessian of $J$})\\
Consider $\Omega$ as in \eqref{defOmega}. Let $M:= \left(S_{d}^{-\frac{1}{2}}(\Omega\tilde{S}\Omega\t)^{-1}S_{d}^{-\frac{1}{2}}\right)^{\!\frac{1}{2}}$, and $\tilde{M}:=S_{d}^{\frac{1}{2}}M S_{d}^{\frac{1}{2}} = S_{d} \# (\Omega\tilde{S}\Omega\t)^{-1}$. Then the Hessian of $J$ in \eqref{eq:total-objective} is
\begin{align}
\hess&(J) = 2\left(\tilde{S}\otimes I\right) + 2\lambda\left(\tilde{S}\otimes \bs{H_u}\t\bs{F}\t\bs{F H_u}\right) \nonumber\\
&+ 2\lambda\left(\Omega\tilde{S}\otimes \bs{F H_u}\right)\t \left(S_{d}^{\frac{1}{2}}M S_{d}^{-\frac{1}{2}}\oplus S_{d}^{\frac{1}{2}}M S_{d}^{-\frac{1}{2}}\right)\nonumber\\
&\left(\!\left(\Omega\tilde{S}\Omega\t\right)^{\!-1}\!\!\otimes\!\left(\Omega\tilde{S}\Omega\t\right)^{\!-1}\!\right)(I+K_{0})\left(\Omega\tilde{S}\otimes \bs{F H_u}\right)\nonumber\\
&-2\lambda\left(\tilde{S}\otimes\bs{H_u}\t\bs{F}\t \tilde{M}\bs{F H_u}\right).
\label{HessJThmFormula}    
\end{align}
\end{theorem}
\begin{proof}
Combining \eqref{HessianOfJ}, \eqref{HessOfJ2}, \eqref{HessOfJ3} and \eqref{IdentifyingHessJ4}, we get
\begin{align}
\hess(J) = &\:2\left(\tilde{S}\otimes I\right) + 2\lambda\left(\tilde{S}\otimes \bs{H_u}\t\bs{F}\t\bs{F H_u}\right)\nonumber\\
&-2\lambda\left(\tilde{S}\otimes\bs{H_u}\t\bs{F}\t\right)\jac P\left(\Omega(\bm{\Theta})\right)\jac\Omega(\bm{\Theta}).  
\label{HessJintermed}    
\end{align}
From \eqref{defOmega} and Lemma \ref{JacOfAXB}, we have $\jac\Omega(\bm{\Theta}) = I\otimes\bs{F H_u}$. From \eqref{jacPfromjacQ} and \eqref{jacQfinal}, we also have
\begin{align*}
&\jac P(\Omega) = \left(\Omega\t\otimes I\right)\jac Q(\Omega) + (I\otimes Q(\Omega))\\
&= -\left(\Omega\t\otimes I\right)\!\left(\!S_{d}^{\frac{1}{2}}\otimes S_{d}^{\frac{1}{2}}\!\right)\!(M\oplus M)\!\left(\!S_{d}^{-\frac{1}{2}}\otimes S_{d}^{-\frac{1}{2}}\!\right)\\
&\left(\!\left(\Omega\tilde{S}\Omega\t\right)^{\!-1}\!\!\otimes\!\left(\Omega\tilde{S}\Omega\t\right)^{\!-1}\right)(I+K_{0})(\Omega\tilde{S}\otimes I) + (I\otimes \tilde{M})\\
&\stackrel{\eqref{KronSumProductSimilarityTransform}}{=} -\left(\Omega\t\otimes I\right)\!\left(\!S_{d}^{\frac{1}{2}}M S_{d}^{-\frac{1}{2}}\oplus S_{d}^{\frac{1}{2}}M S_{d}^{-\frac{1}{2}}\right)\\
&\left(\!\left(\Omega\tilde{S}\Omega\t\right)^{\!-1}\!\!\otimes\!\left(\Omega\tilde{S}\Omega\t\right)^{\!-1}\right)(I+K_{0})(\Omega\tilde{S}\otimes I) + (I\otimes \tilde{M}).
\end{align*}
In \eqref{HessJintermed}, substituting the above for $\jac P(\Omega)$ and $\jac\Omega(\bm{\Theta})$, and then substituting $\Omega$ as function of $\bm{\Theta}$ from \eqref{defOmega}, yields \eqref{HessJThmFormula} as claimed.
\end{proof}
With the formula \eqref{HessJThmFormula} in hand, $\hess(J) \succ \bm{0}$ (strictly positive definite) is a sufficient condition for the unique solution for $\bm{\Theta}^{\star}$ (since constraints \eqref{CausalityConstr} are linear). In Theorem \ref{thm:suffconditionHessPosDef} below, we deduce a simpler sufficient condition involving the terminal covariance (i.e., covariance of the random vector $x(N)$) that guarantees $\hess(J) \succ \bm{0}$. It was shown in \cite[Sec. III]{p:balci2020covariance} that the covariance of $x(N)$ equals $\Omega\tilde{S}\Omega\t$.

\begin{theorem}\label{thm:suffconditionHessPosDef}
If the terminal covariance $\Omega\tilde{S}\Omega\t \succeq S_{d}$, then $\hess(J)\succ\bm{0}$.
\end{theorem}
\begin{proof}
Let us view the right-hand-side of \eqref{HessJThmFormula} as linear combination of four terms. We refer to $2\left(\tilde{S}\otimes I\right)$ as term 1, the quantity $2\lambda\left(\tilde{S}\otimes \bs{H_u}\t\bs{F}\t\bs{F H_u}\right)$ as term 2, and $2\lambda\left(\tilde{S}\otimes\bs{H_u}\t\bs{F}\t \tilde{M}\bs{F H_u}\right)$ as term 4. The remaining term in \eqref{HessJThmFormula} is referred to as term 3.

We now make use of two instances of the L\"{o}wner–Heinz theorem \cite{lowner1934monotone,heinz1951beitrage,kato1952notes}; see also \cite[Sec. 2]{carlen2010trace}. Specifically, since $X \mapsto X^{-1}$ is operator decreasing, we have 
\begin{align}
\Omega\tilde{S}\Omega\t \succeq S_{d}  &\Rightarrow\:  \left(\Omega\tilde{S}\Omega\t\right)^{-1} \preceq S_{d}^{-1}  \nonumber\\
&\Rightarrow\:  S_{d}^{-\frac{1}{2}}\left(\Omega\tilde{S}\Omega\t\right)^{-1}S_{d}^{-\frac{1}{2}} \preceq S_{d}^{-2},\label{CongruenceTransform}
\end{align}
where the last line follows from the congruence transform by $S_{d}^{-\frac{1}{2}}$. On the other hand, since $X\mapsto X^{\frac{1}{2}}$ is operator increasing, \eqref{CongruenceTransform} gives
\begin{align}
\underbrace{\left(S_{d}^{-\frac{1}{2}}\left(\Omega\tilde{S}\Omega\t\right)^{-1}S_{d}^{-\frac{1}{2}}\right)^{\frac{1}{2}}}_{=M\text{ (defined in Thm. \ref{thm:Hessian})}} \preceq S_{d}^{-1} \Rightarrow \!\!\!\!\underbrace{S_{d}^{\frac{1}{2}}MS_{d}^{\frac{1}{2}}}_{=\tilde{M}\text{ (defined in Thm. \ref{thm:Hessian})}} \!\!\preceq I,
\label{MtildeLessThanIdentity}
\end{align}
where the last inequality is due to congruence transform by $S_{d}^{\frac{1}{2}}$. Since $\tilde{S}\succ\bm{0}$, it follows from \eqref{MtildeLessThanIdentity} that 
\[\tilde{S}\otimes\bs{H_u}\t\bm{F}\t\tilde{M}\bm{F}\bs{H_u} \preceq \tilde{S}\otimes\bs{H_u}\t\bm{F}\t\bm{F}\bs{H_u},\]
and multiplying both sides of the above by $-2\lambda<0$, we get
\begin{align}
-\:\text{term 4} \succeq -\text{term 2}.\label{term4ineq}    
\end{align}

Since $\tilde{S}\succ\bm{0}$ (see Proposition \ref{prop:tildeSposdef}), we have
\begin{align}
\text{term 1} \succ \bm{0}.\label{term1ineq}    
\end{align}

On the other hand, since $M\succ\bm{0}$, the similarity transform $S_{d}^{\frac{1}{2}}MS_{d}^{-\frac{1}{2}}\succ\bm{0}$, and therefore, $S_{d}^{\frac{1}{2}}MS_{d}^{-\frac{1}{2}}\oplus S_{d}^{\frac{1}{2}}MS_{d}^{-\frac{1}{2}}\succ\bm{0}$. Furthermore, since $(\Omega\tilde{S}\Omega\t)^{-1}\succ\bm{0}$, we have $(\Omega\tilde{S}\Omega\t)^{-1}\otimes (\Omega\tilde{S}\Omega\t)^{-1}\succ\bm{0}$. Because the product of positive definite matrices is positive definite, we thus get
\[\underbrace{\left(\!S_{d}^{\frac{1}{2}}MS_{d}^{-\frac{1}{2}}\oplus S_{d}^{\frac{1}{2}}MS_{d}^{-\frac{1}{2}}\!\right)\!\left(\!(\Omega\tilde{S}\Omega\t)^{-1}\otimes (\Omega\tilde{S}\Omega\t)^{-1}\!\right)}_{=:T}\succ\bm{0}.\]
Recall from Sec. \ref{s:prelim} that the matrix $I+K_{0}$ is positive semidefinite (has eigenvalues 0 and 2). With matrix $T\succ\bm{0}$ defined as above, notice that the spectrum of $T(I+K_{0})$ is identical to the spectrum of $\sqrt{T}(I+K_{0})\sqrt{T}\succeq\bm{0}$. So, $T(I+K_{0})\succeq\bm{0}$, and consequently,
\begin{align}
\underbrace{2\lambda\left(\Omega\tilde{S}\otimes\bs{F H_u}\right)\t T(I+K_{0})\left(\Omega\tilde{S}\otimes\bs{F H_u}\right)}_{=\:\text{term 3}} \succeq \bm{0}.
\label{term3ineq}     
\end{align}
From \eqref{HessJThmFormula}, we arrive at
\begin{align*}
\hess(J) &= \text{term 1} + \text{term 2} + \text{term 3} - \text{term 4}\\
& \stackrel{\eqref{term4ineq}}{\succeq} \text{term 1} + \text{term 3}\\
& \stackrel{\eqref{term1ineq}, \eqref{term3ineq}}{\succ}\bm{0},    
\end{align*}
completing the proof.
\end{proof}


\section{Conclusions}\label{s:conclusion}
In this paper, we analyzed the covariance steering problem for discrete time Gaussian linear systems with a squared Wasserstein distance terminal cost focusing on the existence and the uniqueness of the solution to the problem. We showed that this problem is in general nonconvex, and may admit more than one local minimizers. We also derived the analytical expression of the Jacobian and the Hessian of the objective function based on specialized tools from matrix calculus, and obtained the first-order and second-order conditions for optimality. Finally, we presented a sufficient condition for the strict convexity of the performance index, thereby guaranteeing the uniqueness of the solution to the optimization problem under the same condition. This sufficient condition is particularly appealing: the desired state covariance is upper bounded (in L\"{o}wner sense) by the terminal state covariance. The analysis of the convergence rate of the convex-concave procedure, and the development of new computational schemes exploiting the derived first-order and second-order conditions for optimality, will be explored in our future research.


\appendix
We collect few lemmas on the Jacobians of some matrix functions which are used in this paper.
\begin{lemma}\label{JacOfAXB}
Let $F(X) := AXB$. Then $\jac F(X) = B\t\otimes A$.
\end{lemma}
\begin{proof}
See \cite[Lemma 3 in Appendix]{halder2016finite}.
\end{proof}
\begin{lemma}\label{JacOfXXt}
Let $F(X) := XX\t$. Then $\jac F(X) = (I + K_{0})\left(X\otimes I\right)$.
\end{lemma}
\begin{proof}
See \cite[Ch. 9, Sec. 14]{magnus2019matrix}.
\end{proof}
\begin{lemma}\label{JacOfXSXt}
Let $F(X) := X\tilde{S}X\t$ where $\tilde{S}$ is a given symmetric positive definite matrix. Then $\jac F(X) = (I + K_{0})\left(X\tilde{S}\otimes I\right)$.
\end{lemma}
\begin{proof}
Let $F_{1}(X) = XX\t$, and $F_{2}(X) = X\tilde{S}^{\frac{1}{2}}$. Then $F(X) = F_{1}\circ F_{2}(X)$. Hence 
\begin{align*}
\jac F(X) &= \jac F_{1}\left(F_{2}(X)\right)\jac F_{2}(X)\\
&= \underbrace{(I + K_{0})\left(X\tilde{S}^{\frac{1}{2}}\otimes I\right)}_{\text{using Lemma \ref{JacOfXXt}}}\underbrace{\left(\tilde{S}^{\frac{1}{2}}\otimes I\right)}_{\text{using Lemma \ref{JacOfAXB}}}\\ &\stackrel{\eqref{ProductOfKronProduct}}{=} (I + K_{0})\left(X\tilde{S}\otimes I\right).\qedhere   
\end{align*}
\end{proof}
\begin{lemma}\label{JacOfXinv}
For $X$ nonsingular, let $F(X) := X^{-1}$. Then $\jac F(X) = -\left(X^{-{\mathrm{T}}}\otimes X^{-1}\right)$.
\end{lemma}
\begin{proof}
Applying the differential operator $\differential(\cdot)$ to both sides of the identity $XX^{-1}=I$ gives $\left(\differential X\right) X^{-1} + X \left(\differential X^{-1}\right) = \bm{0}$, which upon rearranging yields 
\[\differential X^{-1} = - X^{-1}\left(\differential X\right) X^{-1}.\]
Applying $\vec$ to both sides of the above and then using (\ref{kronANDvec}), results in
\[\differential\:\vec X^{-1} = -\left(X^{-{\mathrm{T}}}\otimes X^{-1}\right) \differential\:\vec\:X.\]
In above, invoking \eqref{JacIdentificationRule} completes the proof.
\end{proof}

\bibliographystyle{ieeetr}
\bibliography{convex-wasserstein}

\end{document}